\documentclass[a4paper]{article}

\usepackage{amsfonts}
\usepackage{amsthm}
\usepackage{amsmath}
\usepackage{amssymb}
\usepackage{xcolor}

\usepackage[margin=2.54cm]{geometry}

\numberwithin{equation}{section}
\newtheorem{theorem}{Theorem}[section]
\newtheorem{lemma}[theorem]{Lemma}

\theoremstyle{remark}

\DeclareMathOperator{\vol}{vol}
\DeclareMathOperator{\sign}{sign}

\title{Local solvability of the Poisson equation for closed $G_2$-structures}

\author{Timothy Buttsworth\thanks{School of Mathematics and Physics, The University of Queensland, St Lucia,~QLD 4072, Australia}~\thanks{Dr Buttsworth is the recipient of an Australian Research Council Discovery Early-Career Researcher Award (DE220100919) funded by the Australian Government.} \\
\small{\texttt{t.buttsworth@uq.edu.au}}
\and Artem Pulemotov\footnotemark[1]~\thanks{Research supported by the Australian Government through the Australian Research Council's Discovery Projects funding scheme (DP220102530).} \\
\small{\texttt{a.pulemotov@uq.edu.au}}}

\begin{document}

\maketitle

\begin{abstract}
\noindent We prove local solvability of the Poisson equation with a positive or negative right-hand side for closed $G_2$-structures.
\end{abstract}

\section{Introduction and the main result}

Consider a smooth oriented manifold $M$ of dimension~7. Given a closed $G_2$-structure $\sigma$ on~$M$, denote its induced Riemannian metric by~$g_\sigma$. Let $\Delta_\sigma$ be the Hodge Laplacian of~$g_\sigma$. The operator
\begin{align}\label{nonlin_Laplacian}
\sigma\mapsto\Delta_\sigma\sigma
\end{align}
plays a major part in the study of special holonomy. It is well-known that $\sigma$ is torsion-free if and only if
\begin{align}\label{Laplace_eq}
\Delta_\sigma\sigma=0.
\end{align}
In this case, the holonomy of $g_\sigma$ lies in~$G_2$. Several authors have investigated the heat flow and the soliton equation corresponding to~\eqref{Laplace_eq}. We refer to the survey~\cite{JL20} for some of the most notable results and open questions on these topics.

As per the classical paradigm, in order to achieve complete understanding of the operator~\eqref{nonlin_Laplacian}, one needs to study the Poisson equation
\begin{align}\label{eq_Poisson}
\Delta_\sigma\sigma=\eta,
\end{align}
where $\eta$ is a given 3-form. In this paper, we obtain an existence theorem for solutions to~\eqref{eq_Poisson} in a neighbourhood of a point~$p\in M$. A similar result for the prescribed Ricci curvaure problem, proven in~\cite{DDT81}, has led to the discovery of the famous DeTurck trick for the Ricci flow. We expect, in the same spirit, that our findings will furnish new insights into the properties of~\eqref{Laplace_eq} and other related equations.

Throughout the paper, the terms \emph{$G_2$-structure} and \emph{form} refer to smooth $G_2$-structures and forms unless indicated otherwise. By definition, a 3-form $\eta$ is positive if it is an oriented $G_2$-structure on~$M$ and negative if $-\eta$ is positive; cf.~\cite[Section~10.1]{DJ00}. Our main result is the following.

\begin{theorem}\label{thm_main}
Assume that $\eta$ is a positive or negative closed 3-form on~$M$. Every point $p\in M$ has a neighbourhood $U$ such that there exists a closed oriented $G_2$-structure $\sigma$ on $U$ satisfying equation~(\ref{eq_Poisson}) on~$U$. 
\end{theorem}

Let us outline the main ideas behind the proof. While equation~\eqref{eq_Poisson} is not elliptic, it can be ``adjusted" by a diffeomorphism to become elliptic in the direction of closed forms. Similar ``adjustments" were used in~\cite{BX11} to establish short-time existence for the heat flow corresponding to~\eqref{Laplace_eq} and in~\cite[Chaper~5]{AB87} to tackle the prescribed Ricci curvature problem. The proof of Theorem~\ref{thm_main} requires two separate applications of this trick: one to solve equation~\eqref{eq_Poisson} at the point~$p$ and another to make this equation tractable in a neighbourhood of~$p$. After~\eqref{eq_Poisson} has been ``adjusted", we need to leverage its ellipticity in the direction of closed forms to construct a closed $G_2$-structure $\sigma$ that satisfies it. In~\cite{BX11}, a similar task for the heat flow was accomplished by using the Nash--Moser inverse function theorem. However, we take a different approach and adapt the classical local theory of elliptic equations to the setting of closed forms. This strategy relies on the Banach fixed point theorem rather than the Nash--Moser inverse function theorem.

\section{Linear operators on closed forms}\label{sec_cl_forms}

The rest of the paper is devoted to the proof of Theorem~\ref{thm_main}. It will be convenient for us to assume henceforth that $M$ admits a global coordinate system $(x_1,\ldots,x_7)$ centred at~$p$. Since Theorem~\ref{thm_main} only concerns local solvability of equation~(\ref{eq_Poisson}), there is no loss of generality. For an index $i$ and an ordered set of indices $(i_1,\ldots,i_n)$, denote $e_i=\frac\partial{\partial x_i}$ and $e^{i_1\cdots i_n}=dx_{i_1}\wedge\cdots\wedge dx_{i_n}$. Let $g$ be the Euclidean metric on~$M$, i.e., the metric given by the identity matrix in the coordinates $(x_1,\ldots,x_7)$ at every point. 

Our argument requires several general results about linear operators on closed forms. We establish these results in the present section. Consider a bounded contractible closed region ${N\subset M}$ such that $\partial N$ is smooth and $p$ lies in the interior of~$N$. Fix a number $a\in(0,1)$. Given $l\in\mathbb N$ and $m\in\{0,\ldots,7\}$, let $C^{l,a}\Omega^m(N)$ be the $C^{l,a}$-H\"older space of $m$-forms on~$N$. By definition,
\begin{align*}
\|\phi\|_{C^{l,a}\Omega^m(N)}=\sum_{i=0}^l|\nabla^i\phi|_{C^0(N)}+[\,\nabla^l\phi\,]_{C^{0,a}(N)},\qquad \phi\in C^{l,a}\Omega^m(N).
\end{align*}
In this formula, $\nabla$ is the connection on the tensor bundle induced by the Levi-Civita connection of~$g$, $|\cdot|_{C^0(N)}$ is the sup-norm of a tensor field on~$N$ with respect to~$g$, and $[\,\cdot\,]_{C^{0,a}(N)}$ is the H\"older coefficient. Denote
\begin{align*}
C^{l,a}\mathcal X^m(N)&=\{\phi\in C^{l,a}\Omega^m(N)\,|\,d\phi=0\},
\\
C_0^{l,a}\mathcal X^m(N)&=\{\phi\in C^{l,a}\mathcal X^m(N)\,|\,\nabla^i\phi(p)=0\mbox{~for all~}i=0,\ldots,l\}.
\end{align*}
Clearly, these are subspaces of $C^{l,a}\Omega^m(N)$.

Our first result in this section concerns the Hodge Laplacian $\Delta_g$ associated with~$g$. Roughly speaking, we show that $\Delta_g$ has a right inverse taking closed forms to closed forms.

\begin{lemma}\label{lem_Hodge_inv}
There exists a bounded operator
$$R:C_0^{l,a}\mathcal X^3(N)\to C_0^{l+2,a}\mathcal X^3(N)
$$
such that ${\Delta_gR\phi=\phi}$ for all $\phi\in C_0^{l,a}\mathcal X^3(N)$.
\end{lemma}

\begin{proof}
Let $\nu$ be the outward unit normal vector field on~$\partial N$. Denote by $\delta$ the co-differential with respect to~$g$. According to the Hodge--Morrey decomposition theorem (see~\cite[Theorem~6.12 and Remark~6.13]{CDK12}), the space $C^{l+1,a}\Omega^2(N)$ splits into a direct sum as follows:
$$C^{l+1,a}\Omega^2(N)=C^{l+1,a}\mathcal E^2(N)\oplus C^{l+1,a}\mathcal C^2(N)\oplus C^{l+1,a}\mathcal{H}^2(N).$$
On the right-hand side,
\begin{align*}
C^{l+1,a}\mathcal E^2(N)&=\{d\psi\,|\,\psi\in C^{l+2,a}\Omega^1(N)~\mbox{and}~\nu\lrcorner\psi_{|\partial N}=0\}, \\
C^{l+1,a}\mathcal C^2(N)&=\{\delta\psi\,|\,\psi\in C^{l+2,a}\Omega^3(N)~\mbox{and}~\nu\lrcorner\psi_{|\partial N}=0\},\\
C^{l+1,a}\mathcal{H}^2(N)&=\{\psi\in C^{l+1,a}\Omega^2(N) \,|\, d\psi=0,~\delta\psi=0~\mbox{and}~\nu\lrcorner \psi_{|\partial N}=0\}.
\end{align*}
Since $N$ is contractible, $C^{l+1,a}\mathcal{H}^2(N)=\{0\}$ by~\cite[Theorem~6.5]{CDK12}. The kernel of the exterior derivative operator
\begin{align*}
d:C^{l+1,a}\Omega^2(N)\to C^{l,a}\mathcal X^3(N)
\end{align*}
coincides with $C^{l+1,a}\mathcal E^2(N)$. Indeed, if $d\tau=0$ and $\tau\in C^{l+1,a}\mathcal C^2(N)$, then $\tau$ is both closed and co-closed. Moreover, $\nu\lrcorner\tau_{|\partial N}=0$ by~\cite[Theorem~3.23]{CDK12}. Using the contractibility of $N$ one more time, we conclude that $\tau=0$. Thus, the kernel of $d$ must equal $C^{l+1,a}\mathcal E^2(N)$. This fact and the optimal regularity variant of the Poincar\'e lemma (see~\cite[Section~8.2]{CDK12}) imply that the restriction 
$$
d_{|C^{l+1,a}\mathcal C^2(N)}:C^{l+1,a}\mathcal C^2(N)\to C^{l,a}\mathcal X^3(N)
$$ is an isomorphism. 

Given a 2-form $\beta\in C^{l+1,a}\mathcal C^2(N)$, consider the boundary-value problem 
$$d\alpha=0,\qquad \delta\alpha=\beta, \qquad \nu^\flat\wedge\alpha_{|\partial N}=0.$$
According to~\cite[Theorem~7.2]{CDK12}, this problem has a unique solution $\alpha\in C^{l+2,a}\mathcal X^3(N)$. Define
$$
Q_1:C^{l+1,a}\mathcal C^2(N)\to C^{l+2,a}\mathcal X^3(N)
$$
to be the operator that sends $\beta$ to this solution. Invoking~\cite[Theorem~7.2]{CDK12} again, we conclude that $Q_1$ bounded.

Let $\alpha\in C^{l+2,a}\mathcal X^3(N)$ be a 3-form with components $(\alpha_{ijk})_{i,j,k=1}^7$ in the coordinates $(x_1,\ldots,x_7)$. Given $n\in\{1,\ldots,l+2\}$ and a $C^n$-differentiable function $f:N\to\mathbb R$, denote by $\mathcal P_n(f)$ the $n$th Taylor polynomial of~$f$ at the point~$p$ in these coordinates. Set
\begin{align*}
\mathcal P_n(\alpha)=\sum_{i,j,k=1}^7\mathcal P_n(\alpha_{ijk})e^{ijk}.
\end{align*}
The 3-form $\mathcal P_n(\alpha)$ must be closed. Indeed, every component of $d\mathcal P_n(\alpha)$ is the $(n-1)$th Taylor polynomial of the corresponding component of~$d\alpha$.
Define
$$
Q_2:C^{l+2,a}\mathcal X^3(N)\to C_0^{l+2,a}\mathcal X^3(N)
$$
to be the operator that sends $\alpha$ to $\alpha-\mathcal P_{l+2}(\alpha)$. Clearly, $Q_2$ is bounded.

Set $R=Q_2Q_1d_{|C^{l+1,a}\mathcal C^2(N)}^{-1}$. For $\phi\in C_0^{l,a}\mathcal X^3(N)$,
\begin{align*}
\Delta_gR\phi&=\Delta_gQ_1d_{|C^{l+1,a}\mathcal C^2(N)}^{-1}\phi-\Delta_g\mathcal P_{l+2}\big(Q_1d_{|C^{l+1,a}\mathcal C^2(N)}^{-1}\phi\big).
\end{align*}
We have
\begin{align*}
\Delta_gQ_1d_{|C^{l+1,a}\mathcal C^2(N)}^{-1}\phi=d\delta Q_1d_{|C^{l+1,a}\mathcal C^2(N)}^{-1}\phi=dd_{|C^{l+1,a}\mathcal C^2(N)}^{-1}\phi=\phi.
\end{align*}
To prove the lemma, we need to show that
\begin{align}\label{eq_Taylor_van}
\Delta_g\mathcal P_{l+2}\big(Q_1d_{|C^{l+1,a}\mathcal C^2(N)}^{-1}\phi\big)=0.
\end{align}

Denote $\theta=Q_1d_{|C^{l+1,a}\mathcal C^2(N)}^{-1}\phi$. The components of the 3-form $\Delta_g\mathcal P_{l+2}(\theta)$ in the coordinates $(x_1,\ldots,x_7)$ are polynomials in $x_1,\ldots,x_7$ of degree up to~$l$. Since
\begin{align*}
\nabla^n\Delta_g\mathcal P_{l+2}(\theta)(p)=\nabla^n\Delta_g\theta(p)=\nabla^n\phi(p)=0
\end{align*}
for every $n\in\{0,\ldots,l\}$, these polynomials vanish at~$p$. This implies~\eqref{eq_Taylor_van}.
\end{proof}

Given $s>0$, let $N_s$ be the closed geodesic ball with respect to $g$ of radius~$s$ centred at~$p$. Consider a second-order linear differential operator $K$ acting on 3-forms on~$M$. Assume that the coefficients of $K$ with respect to $(x_1,\ldots,x_7)$ are smooth. Our next goal is to study the invertibility of
$$
L=\gamma\Delta_g+K
$$
with $\gamma>0$. In our reasoning, we will exploit the bounded dilation operator
$$
A_{l,s}:C^{l,a}\mathcal X^3(N_1)\to C^{l,a}\mathcal X^3(N_s)
$$
defined for $l\in\mathbb N\cup\{0\}$ by the formula
$$
A_{l,s}\phi(x_1,\ldots,x_7)=\phi(s^{-1}x_1,\ldots,s^{-1}x_7)
$$
in the coordinates $(x_1,\ldots,x_7)$. Clearly, this operator is an isomorphism. The equality
\begin{align}\label{La_comm}
\Delta_gA_{l+2,s}\phi=s^{-2}A_{l,s}\Delta_g\phi
\end{align}
holds for every $\phi\in C^{l+2,a}\mathcal X^3(N_1)$. 

\begin{lemma}\label{lem_lin_inv}
Let the highest-order coefficients of $K$ with respect to $(x_1,\ldots,x_7)$ vanish at~$p$. Assume that the form $L\phi$ is closed for every closed form~$\phi$. If $s$ is sufficiently small, then there exists a bounded operator
$$P:C_0^{l,a}\mathcal X^3(N_s)\to C_0^{l+2,a}\mathcal X^3(N_s)$$
such that $LP\phi=\phi$ for all $\phi\in C_0^{l,a}\mathcal X^3(N_s)$.
\end{lemma}

\begin{proof}
Choose $s\in(0,1)$ and $\phi\in C_0^{l,a}\mathcal X^3(N_s)$. Since both $L$ and $\Delta_g$ take closed forms to closed forms, so does~$K$. Consider the map 
$$
\mathfrak B_s:C_0^{l+2,a}\mathcal X^3(N_1)\to C_0^{l+2,a}\mathcal X^3(N_1)
$$
given by
\begin{align*}
\mathfrak B_s(\psi)=s^2\gamma^{-1}RA_{l,s}^{-1}\phi-s^2\gamma^{-1}RA_{l,s}^{-1}KA_{l+2,s}\psi,\qquad \psi\in C_0^{l+2,a}\mathcal X^3(N_1),
\end{align*}
where $R$ is the right inverse of $\Delta_g$ produced in Lemma~\ref{lem_Hodge_inv}. Recalling that the highest-order coefficients of $K$ vanish at~$p$, one can easily prove that
\begin{align*}
\big\|A_{l,s}^{-1}KA_{l+2,s}\psi\big\|_{C^{l,a}\Omega^3(N_1)}\le cs^{-1}\|\psi\|_{C^{l+2,a}\Omega^3(N_1)},\qquad \psi\in C_0^{l+2,a}\mathcal X^3(N_1),
\end{align*}
for some $c>0$. Therefore, if $s$ is sufficiently small, the map $\mathfrak B_s$ is a contraction. Denote its unique fixed point by $\psi_*$. Shrinking $s$ further if necessary, we find
\begin{align}\label{BFP_bound}
\|\psi_*\|_{C^{l+2,a}\Omega^3(N_1)}\le 2s^2\gamma^{-1}\big\|RA_{l,s}^{-1}\big\|\|\phi\|_{C^{l,a}\Omega^3(N_s)}.
\end{align}
The formula $P\phi=A_{l+2,s}\psi_*$ defines an operator $P$ from $C_0^{l,a}\mathcal X^3(N_s)$ to $C_0^{l+2,a}\mathcal X^3(N_s)$. Inequality~\eqref{BFP_bound} implies that this operator is bounded. 
Using~\eqref{La_comm}, we compute
\begin{align*}
LP\phi&=\gamma\Delta_gA_{l+2,s}\psi_*+KA_{l+2,s}\psi_*
\\
&=s^{-2}\gamma A_{l,s}\Delta_g\big(s^2\gamma^{-1}RA_{l,s}^{-1}\phi-s^2\gamma^{-1}RA_{l,s}^{-1}KA_{l+2,s}\psi_*\big)+KA_{l+2,s}\psi_*=\phi.
\end{align*}
\end{proof}

Finally, we need the following result based on the extension theorem for H\"older spaces. We assume, for convenience, that $s\in(0,1]$.

\begin{lemma}\label{lem_exten}
Consider a $C^{l+1}$-differentiable closed 3-form $\beta$ on $N_s$. Assume that ${\nabla^i\beta(p)=0}$ for all $i\in\{0,\ldots,l\}$. Given $\epsilon>0$, there exist $r\in(0,s)$ and $\hat\beta\in C^{l,a}\mathcal X^3(N_s)$ such that $\beta_{|N_r}=\hat\beta_{|N_r}$ and ${\|\hat\beta\|_{C^{l,a}\Omega^3(N_s)}<\epsilon}$.
\end{lemma}

\begin{proof}
Fix $r\in(0,s)$. Throughout the proof of the lemma, we use the same letter $c$ for several different constants appearing in our estimates. None of these constants depends on $\beta$ or~$r$. It will be convenient for us to denote $\beta_r=\beta_{|N_r}$. The form $A_{l,r}^{-1}\beta_r$ is closed. By the optimal regularity variant of the Poincar\'e lemma (see~\cite[Section~8.2]{CDK12}), there exists $\alpha\in C^{l+1,a}\Omega^2(N_1)$ such that $A_{l,r}^{-1}\beta_r=d\alpha$. Moreover, we may assume that 
\begin{align*}
\|\alpha\|_{C^{l+1,a}\Omega^2(N_1)}\le c \big\|A_{l,r}^{-1}\beta_r\big\|_{C^{l,a}\Omega^3(N_1)}.
\end{align*}
According to~\cite[Theorem~16.11]{CDK12}, there exists a bounded extension operator
\begin{align*}
\iota:C^{l+1,a}\Omega^2(N_1)\to C^{l+1,a}\Omega^2(M).
\end{align*}
Define $\hat\beta=d\hat\alpha$, where $\hat\alpha\in C^{l+1,a}\Omega^2(N_s)$ is given by the formula
\begin{align*}
\hat\alpha(x_1,\ldots,x_7)=(r\iota\alpha)(r^{-1}x_1,\ldots,r^{-1}x_7)
\end{align*}
in the coordinates $(x_1,\ldots,x_7)$. Clearly, $\hat\beta\in C^{l,a}\mathcal X^3(N_s)$ and $\beta_r=\hat\beta_{|N_r}$. It remains to show that $\|\hat\beta\|_{C^{l,a}\Omega^3(N_s)}<\epsilon$ for appropriately chosen~$r$.

Since $\nabla^i\beta$ vanish at~$p$ whenever $i\in\{0,\ldots,l\}$,
\begin{align*}
|\nabla^i\beta_r|_{C^0(N_r)}\le cr^{l+1-i}|\nabla^{l+1}\beta_r|_{C^0(N_r)}\le cr^{l+1-i}|\nabla^{l+1}\beta|_{C^0(N_s)}.
\end{align*}
The H\"older coefficient $[\nabla^l\beta_r]_{C^{0,a}(N_r)}$ satisfies
\begin{align*}
[\nabla^l\beta_r]_{C^{0,a}(N_r)}
\le cr^{1-a}|\nabla^{l+1}\beta|_{C^0(N_s)}.
\end{align*}
As a consequence,
\begin{align*}
\|\hat\beta\|_{C^{l,a}\Omega^3(N_s)}
&\le c\|\hat\alpha\|_{C^{l+1,a}\Omega^2(N_s)}
\\
&\le cr^{-l-a}\|\iota\alpha\|_{C^{l+1,a}\Omega^2(M)}
\le cr^{-l-a}\|\iota\| \big\|A_{l,r}^{-1}\beta_r\big\|_{C^{l,a}\Omega^3(N_1)}
\\
&\le cr^{-l-a}\|\iota\|\bigg(\sum_{i=1}^lr^i|\nabla^i\beta_r|_{C^0(N_r)}+r^{l+a}[\nabla^l\beta_r]_{C^{0,a}(N_r)}\bigg)
\\
&\le cr^{1-a}\|\iota\||\nabla^{l+1}\beta|_{C^0(N_s)}.
\end{align*}
Shrinking $r$ if necessary, we can ensure that $\|\hat\beta\|_{C^{l,a}\Omega^3(N_s)}<\epsilon$.
\end{proof}

\section{Proof of Theorem~\ref{thm_main}}\label{sec_proof}

With the results of Section~\ref{sec_cl_forms} at hand, we are ready to prove our main theorem. The canonical $G_2$-structure associated with the coordinates $(x_1,\ldots,x_7)$ is
\begin{align*}
\sigma_{{\rm can}}=\sum_{(i,j,k)\in I_+}e^{ijk}-\sum_{(i,j,k)\in I_-}e^{ijk},
\end{align*}
where $I_+=\{(1,2,3),(1,4,5),(1,6,7),(2,4,6)\}$ and $I_-=\{(2,5,7),(3,4,7),(3,5,6)\}$. We may choose the system $(x_1,\ldots,x_7)$ so that 
\begin{align}\label{coord_canon}
12^{2/3}\sigma_{{\rm can}}(p)=\sign(\eta)\,\eta(p)
\end{align}
with $\sign(\eta)=1$ if $\eta$ is positive and $\sign(\eta)=-1$ if $\eta$ is negative. 
Our first step is to solve equation~\eqref{eq_Poisson} at the point~$p$. 

\begin{lemma}\label{lem_pointsolve}
There exists a closed oriented $G_2$-structure
\begin{align*}
\sigma_0=\sum_{(i,j,k)\in I_+}\sigma_{0ijk}e^{ijk}-\sum_{(i,j,k)\in I_-}\sigma_{0ijk}e^{ijk},
\end{align*}
defined on a neighbourhood of~$p$, such that 
\begin{align*}
\Delta_{\sigma_0}\sigma_0(p)=\eta(p),\qquad \sigma_0(p)=\sign(\eta)\,\eta(p),
\end{align*}
and every component $\sigma_{0ijk}$ is a quadratic polynomial in $x_1,\ldots,x_7$ without linear terms.
\end{lemma}

\begin{proof}
Assume that $\eta$ is positive. Consider the 3-form
\begin{align}\label{pt_sigma_def}
\theta=\sum_{(i,j,k)\in I_+}\big(1-x_i^2-x_j^2-x_k^2\big)e^{ijk}-\sum_{(i,j,k)\in I_-}\big(1-x_i^2-x_j^2-x_k^2\big)e^{ijk}.
\end{align}
Clearly, it is closed and oriented. We will show that $\Delta_{\theta}\theta(p)=12\sigma_{{\rm can}}(p)$. Rescaling $\theta$ will give us $\sigma_0$ and prove the lemma.

Let us find the Riemannian metric $g_\theta$ associated with $\theta$ in the neighbourhood of $p$ where $\theta$ is positive. We have
\begin{align}\label{metric_form}
6g_{\theta}(X,Y)\vol_\theta=(X\lrcorner\theta)\wedge (Y\lrcorner\theta)\wedge\theta,\qquad X,Y\in T_xM,~x\in M,
\end{align}
where $X$ and $Y$ are tangent to $M$ and $\vol_\theta$ is the volume form of~$g_\theta$. It is easy to see that $g_\theta$ is diagonal in the coordinates $(x_1,\ldots,x_7)$ and coincides with the Euclidean metric $g$ at~$p$. Each component $g_{\theta ii}$ satisfies
\begin{align*}
g_{\theta ii}(x)=1+l_i(x)+s_i(x)+O(|x|^3), \qquad x\to p,
\end{align*}
where
\begin{align*}
x=(x_1,\ldots,x_7),\qquad |x|=\sqrt{x_1^2+\cdots+x_7^2},
\end{align*}
and $l_i(x)$ and $s_i(x)$ are homogeneous polynomials of degrees~1 and~2, respectively. Taking determinants on both sides of~\eqref{metric_form}, we conclude that $$
l_1(x)+\cdots+l_7(x)=0
$$
for all~$x$. This fact and~\eqref{metric_form} imply that each $l_i$ must be identically~0. 

Observe that
$$\vol_\theta=\Big(1+\frac{s_1+\cdots+s_7}{2}+O(\left|x\right|^3)\Big)e^{1234567}.$$
Consequently,
$$6g_{\theta ii}\vol_\theta=(6+6s_i+3(s_1+\cdots+s_7)+O(\left|x\right|^3))e^{1234567}.$$
At the same time, direct computation shows that
$$
(e_i\lrcorner\theta)\wedge(e_i\lrcorner\theta)\wedge\theta=(6(1-x_1^2-\cdots-x_7^2-2x_i^2)+O(|x|^3))e^{1234567}.
$$
Substituting into~\eqref{metric_form} and summing over $i$, we conclude that
\begin{align*}
s_1+\cdots+s_7=-2(x_1^2+x_2^2+\cdots+x_7^2).
\end{align*}
Using~\eqref{metric_form} again yields $s_i(x)=-2x_i^2$ and
\begin{align*}
g_{\theta ii}(x)=1-2x_i^2+O(|x|^3).
\end{align*}

Our next step is to calculate the action of the Laplacian $\Delta_\theta$ on~$\theta$. 
Since $\theta$ is closed,
\begin{align*}
\Delta_\theta\theta=-d\star_\theta d\star_\theta\theta,
\end{align*} 
where $\star_\theta$ is the Hodge star operator of~$g_\theta$. We have
\begin{align*}
\star_\theta e^{ijklm}= \sqrt{\frac{g_{\theta ii}g_{\theta jj}g_{\theta kk}g_{\theta ll}g_{\theta mm}}{g_{\theta pp}g_{\theta qq}}}e^{pq},\qquad \star_\theta e^{ijk}= \sqrt{\frac{g_{\theta ii}g_{\theta jj}g_{\theta kk}}{g_{\theta aa}g_{\theta bb}g_{\theta cc}g_{\theta dd}}}e^{abcd}.
\end{align*}
The indices $p$, $q$, $a$, $b$, $c$ and $d$ on the right-hand sides are chosen so that $(i,j,k,l,m,p,q)$ and $(i,j,k,a,b,c,d)$ are even permutations of $\{1,2,3,4,5,6,7\}$. As a consequence,
\begin{align*}
\star_\theta e^{ijklm}&=(1+O(|x|^2))e^{pq}, \\
\star_\theta e^{ijk}&=(1-x_i^2-x_j^2-x_k^2+x_a^2+x_b^2+x_c^2+x_d^2+O(|x|^3))e^{abcd}.
\end{align*}
Together with the fact that $p$ is the origin in the coordinates~$(x_1,\ldots,x_7)$, this implies
\begin{align*}
d\star_\theta d\star_\theta\theta\,(p)=d\star d\star_\theta\theta\,(p),
\end{align*}
where $\star$ is the Hodge star operator of the Euclidean metric~$g$. Also,
\begin{align*}
d\star_\theta\theta&=d\star\sum_{(i,j,k)\in I_+}\big(\tilde\theta_{ijk}+x_a^2+x_b^2+x_c^2+x_d^2+O(|x|^3)\big)e^{ijk}
\\
&\hphantom{=}~-d\star\sum_{(i,j,k)\in I_-}\big(\tilde\theta_{ijk}+x_a^2+x_b^2+x_c^2+x_d^2+O(|x|^3)\big)e^{ijk}
\\
&=d\star\sum_{(i,j,k)\in I_+}\big(\tilde\theta_{ijk}+O(|x|^3)\big)e^{ijk}
-d\star\sum_{(i,j,k)\in I_-}\big(\tilde\theta_{ijk}+O(|x|^3)\big)e^{ijk}
\end{align*}
with $\tilde\theta_{ijk}=1-2x_i^2-2x_j^2-2x_k^2$. We conclude that
\begin{align*}
d\star_\theta d\star_\theta\theta\,(p)&=
d\star d\star\bigg(\sum_{(i,j,k)\in I_+}\tilde\theta_{ijk}e^{ijk}
-\sum_{(i,j,k)\in I_-}\tilde\theta_{ijk}e^{ijk}\bigg)(p)
\\
&=\Delta_g\bigg(\sum_{(i,j,k)\in I_+}\tilde\theta_{ijk}e^{ijk}
-\sum_{(i,j,k)\in I_-}\tilde\theta_{ijk}e^{ijk}\bigg)(p)=12\sigma_{{\rm can}}(p).
\end{align*}

Define $\sigma_0=12^{2/3}\theta$. One easily sees that
\begin{align*}
\Delta_{\sigma_0}\sigma_0(p)=12^{-1/3}\Delta_\theta\theta(p)=12^{2/3}\sigma_{{\rm can}}(p)=\eta(p)
\end{align*}
and $\sigma_0(p)=\eta(p)$.

Assume that $\eta$ is negative. Define $\theta$ by formula~\eqref{pt_sigma_def} with the signs reversed at $x_i^2$, $x_j^2$ and~$x_k^2$. Repeating the argument above with obvious modifications yields the desired result.
\end{proof}

The proof of Theorem~\ref{thm_main} uses a variant of the DeTurck trick; cf.~\cite[Chapter~5]{AB87}. In order to deal with the non-ellipticity of equation~\eqref{eq_Poisson}, we replace the 3-form $\eta$ on the right-hand side with the pullback of $\eta$ by a diffeomorphism. More precisely,
given closed $G_2$-structures $\zeta$ and $\phi$ on a closed geodesic ball~$N_s$ with $s>0$, let $T(\zeta,\phi)$ be the section of $TN_s\otimes T^*N_s\otimes T^*N_s$ defined by
$$
T(\zeta,\phi)=\nabla^{\zeta}-\nabla^\phi.
$$
Here, $\nabla^{\zeta}$ and $\nabla^\phi$ are the connections induced on the tensor bundle by the Levi-Civita connections of $g_{\zeta}$ and~$g_\phi$. Let $(T_{ij}^k(\zeta,\phi))_{i,j,k=1}^7$ be the components of $T(\zeta,\phi)$ in the coordinates $(x_1,\ldots,x_7)$. Consider the vector field
\begin{align*}
V(\zeta,\phi)=\sign(\eta)\sum_{i,j,k=1}^7\Big(\frac{15}{28}g_\phi^{ij}T_{ij}^k(\zeta,\phi)+\frac14g_\phi^{ik}T_{ji}^j(\zeta,\phi)\Big)e_k,
\end{align*}
first introduced in~\cite[Section~2]{BX11}. Denote by~$\varphi_{V(\zeta,\phi)}$ the flow of $V(\zeta,\phi)$ at time~1. Clearly, $\varphi_{V(\zeta,\phi)}$ is orientation-preserving. Moreover, $V(\zeta,\zeta)=0$, and $\varphi_{V(\zeta,\zeta)}$ is the identity. We will construct 3-forms $\sigma_1$ and $\sigma_*$ such that
\begin{align}\label{eq_DeT=0}
\Delta_{\sigma_*}\sigma_*=\varphi^*_{V(\sigma_1,\sigma_*)}\eta.
\end{align}
In light of the diffeomorphism-invariance of the Laplacian, setting
\begin{align}\label{sigma_gauged_final}
\sigma=\big(\varphi^{-1}_{V(\sigma_1,\sigma_*)}\big)^*\sigma_*
\end{align}
will complete the proof of Theorem~\ref{thm_main}.

Let us linearise the operator on the left-hand side of~\eqref{eq_Poisson} near the $G_2$-structure~$\phi$. As shown in~\cite[Section~2]{BX11},
for every closed 3-form~$\psi$ on~$N_s$,
\begin{align}\label{lin_Lapl}
\frac\partial{\partial t}\Delta_{\phi+t\psi}(\phi+t\psi)|_{t=0}=-\Delta_{\phi}\psi+\sign(\eta)\,d(V'_{\zeta,\phi}(\psi)\lrcorner\phi)+\Psi_{\zeta,\phi}(\psi).
\end{align}
In this formula,
\begin{align*}
V_{\zeta,\phi}'(\psi)=\frac\partial{\partial t}V(\zeta,\phi+t\psi)|_{t=0}.
\end{align*}
The components of $V_{\zeta,\phi}'(\psi)$ in the coordinates $(x_1,\ldots,x_7)$ at a point $x$ are linear functions of the components of $\psi(x)$ and~$\nabla\psi(x)$ with coefficients determined by $\zeta$ and~$\phi$. The map $\Psi_{\zeta,\phi}$ takes closed 3-forms to closed 3-forms. The components of $\Psi_{\zeta,\phi}(\psi)$ at $x$ depend linearly on~$\psi(x)$ and~$\nabla\psi(x)$. Even though $\zeta$ appears in~\eqref{lin_Lapl}, the sum on the right-hand side is the same for all choices of~$\zeta$.

Our next step is to produce~$\sigma_1$. We need this 3-form to solve equation~\eqref{eq_Poisson} to first order at the point~$p$. Our construction will require another application of the DeTurck trick. We carry it out in Lemma~\ref{lem_h3} below. Choose numbers $s,q\in(0,1)$. Let $B_q$ be the closed ball in $C_0^{3,a}\mathcal X^3(N_s)$ of radius $q$ centred at~0. In what follows, we fix a $G_2$-structure $\sigma_0$ as in Lemma~\ref{lem_pointsolve}. 
We may assume that $s$ and $q$ are such that the 3-form $\sigma_0+\psi$ is well-defined and positive whenever $\psi$ lies in~$B_q$. Taylor's theorem, equality~\eqref{lin_Lapl} and Cartan's formula imply
\begin{align}\label{eq_DeTurck}
\Delta_{\sigma_0+\psi}(\sigma_0&+\psi)-\varphi^*_{V(\sigma_0,\sigma_0+\psi)}\eta \notag
\\
&=\Delta_{\sigma_0}\sigma_0-\eta-\Delta_{\sigma_0}\psi
+\mathcal L_{V'_{\sigma_0,\sigma_0}(\psi)}(\sign(\eta)\sigma_0-\eta)
+\Psi_{\sigma_0,\sigma_0}(\psi)+J_{\sigma_0,\sigma_0}(\psi)
\end{align}
for $\psi\in B_q$. On the right-hand side, $J_{\sigma_0,\sigma_0}$ is a map from $B_q$ to $C_0^{1,a}\mathcal X^3(N_s)$ satisfying
\begin{align}\label{J_pw_bd}
|J_{\sigma_0,\sigma_0}(\psi)(x)|_g+|\nabla J_{\sigma_0,\sigma_0}(\psi)(x)|_g&\le c(|\psi(x)|_g+|\nabla \psi(x)|_g+|\nabla^2\psi(x)|_g)^2\notag
\\
&\hphantom{=}~+c(|\psi(x)|_g+|\nabla \psi(x)|_g+|\nabla^2\psi(x)|_g)|\nabla^3\psi(x)|_g
\end{align}
for $x\in N_s$ with $c$ independent of $s$ and~$q$ (the notation $|\cdot|_g$ here stands for the norm of a tensor induced by the Euclidean metric~$g$).

\begin{lemma}\label{lem_h3}
There exists a closed oriented $G_2$-structure $\sigma_1$, defined on a neighbourhood of~$p$, such that 
\begin{align}\label{first_order_at_p}
\Delta_{\sigma_1}\sigma_1(p)=\eta(p),\qquad \nabla\Delta_{\sigma_1}\sigma_1(p)=\nabla\eta(p),\qquad \sigma_1(p)=\sign(\eta)\,\eta(p).
\end{align}
\end{lemma}

\begin{proof}
By the Poincar\'e lemma, we can find a 2-form $\tau$ such that $\Delta_{\sigma_0}\sigma_0-\eta=d\tau$. Let $(\tau_{ij})_{i,j=1}^7$ be the components of $\tau$ in the coordinates $(x_1,\ldots,x_7)$. Define
\begin{align*}
\tau_*=\sum_{i,j=1}^7\sum_{k,l=1}^7\frac2{1+3\delta_k^l}\frac{\partial^2\tau_{ij}}{\partial x_k\partial x_l}(p)x_k^3x_le^{ij},
\end{align*}
where $\delta_k^l$ is the Kronecker symbol. It will be clear from the arguments below that the second derivatives of $\tau$ coincide with those of $\Delta_{\sigma_0}\tau_*$ at~$p$. We will demonstrate that the $G_2$-structure $\sigma_{1*}=\sigma_0+d\tau_*$ solves the equations
\begin{align}\label{eq_pw_or1_DeT}
\Delta_{\sigma_{1*}}\sigma_{1*}(p)=\varphi^*_{V(\sigma_0,\sigma_{1*})}\eta(p),\qquad \nabla\Delta_{\sigma_{1*}}\sigma_{1*}(p)=\nabla\varphi^*_{V(\sigma_0,\sigma_{1*})}\eta(p).
\end{align}
Pulling back $\sigma_{1*}$ by $\varphi_{V(\sigma_0,\sigma_{1*})}^{-1}$ will produce~$\sigma_1$ and prove the lemma.

The definition of $\sigma_0$ yields 
$$(\Delta_{\sigma_0}\sigma_0-\eta)(p)=0.$$
Recall that $p$ is the origin in the coordinates~$(x_1,\ldots,x_7)$. Since $\Delta_{\sigma_0}$ is a second-order linear differential operator and the components of $d\tau_*$ are cubic, the 3-form $\Delta_{\sigma_0}d\tau_*$ vanishes at~$p$. For similar reasons,
\begin{align*}
\mathcal L_{V'_{\sigma_0,\sigma_0}(d\tau_*)}(\sign(\eta)\sigma_0-\eta)(p)=0
\end{align*}
and $\Psi_{\sigma_0,\sigma_0}(d\tau_*)(p)=0$. Estimate~\eqref{J_pw_bd} implies that $J_{\sigma_0,\sigma_0}(d\tau_*)(p)=0$. These observations and~\eqref{eq_DeTurck} yield the first equality in~\eqref{eq_pw_or1_DeT}.

Since $\sigma_0(p)=\sign(\eta)\,\eta(p)$ and the components of $d\tau_*$ are cubic,
\begin{align*}
\nabla(\mathcal L_{V_{\sigma_0,\sigma_0}'(d\tau_*)}(\sign(\eta)\sigma_0-\eta))(p)=0
\end{align*}
and $\nabla(\Psi_{\sigma_0,\sigma_0}(d\tau_*))(p)=0$. We use~\eqref{J_pw_bd} as above to conclude that $\nabla J_{\sigma_0,\sigma_0}(d\tau_*)(p)=0$. Formula~\eqref{coord_canon}
implies that the leading coefficients of $\Delta_{\sigma_0}$ coincide with those of the operator $12^{-1}\Delta_g$ at the point~$p$. Consequently,
\begin{align*}
\nabla&\big(\Delta_{\sigma_{1*}}\sigma_{1*}-\varphi_{V(\sigma_0,\sigma_{1*})}^*\eta\big)(p)
=\nabla(\Delta_{\sigma_0}\sigma_0-\eta-\Delta_{\sigma_0}d\tau_*)(p)
=\nabla d(\tau-\Delta_{\sigma_0}\tau_*)(p)
\\
&=
\nabla d\bigg(\sum_{i,j=1}^7\sum_{k,l=1}^7\frac{\partial^2\tau_{ij}}{\partial x_k\partial x_l}(p)\bigg(\frac1{1+\delta_k^l}x_kx_l-\frac1{6(1+3\delta_k^l)}\sum_{m=1}^7\frac{\partial^2}{\partial x_m^2}x_k^3x_l\bigg)e^{ij}\bigg)(p)=0.
\end{align*}
Thus, $\sigma_{1*}$ satisfies the second equality in~\eqref{eq_pw_or1_DeT}.

Set $\sigma_1=\big(\varphi_{V(\sigma_0,\sigma_{1*})}^{-1}\big)^*\sigma_{1*}$. By definition, $V(\sigma_0,\sigma_{1*})$ involves no derivatives of $\sigma_{1*}$ of order above one. Since the components of $d\tau_*$ are cubic and $V(\sigma_0,\sigma_0)=0$, each component of $V(\sigma_0,\sigma_{1*})$ equals $O(|x|^2)$ as $x\to p$. As a consequence, the map $\varphi_{V(\sigma_0,\sigma_{1*})}$ does not move the point~$p$ and is, in fact, a diffeomorphism between neighbourhoods of~$p$. We conclude that
\begin{align*}
\Delta_{\sigma_1}\sigma_1(p)&=\big(\varphi_{V(\sigma_0,\sigma_{1*})}^{-1}\big)^*\Delta_{\sigma_{1*}}\sigma_{1*}(p)=\eta(p).
\end{align*}
A similar but slightly lengthier computation shows that the second formula in~\eqref{first_order_at_p} must hold. Differentiating the flow equation, we find that the derivative of $\varphi_{V(\sigma_0,\sigma_{1*})}$ at $p$ is the identity. This implies
\begin{align*}
\sigma_1(p)=\sigma_{1*}(p)=\sigma_0(p)=\sign(\eta)\,\eta(p).
\end{align*}
\end{proof}

Fix a $G_2$-structure $\sigma_1$ as in Lemma~\ref{lem_h3}. Shrinking $s$ and $q$ if necessary, we may assume that the 3-form $\sigma_1+\psi$ is well-defined and positive whenever $\psi$ lies in~$B_q$. By analogy with equation~\eqref{eq_DeTurck}, we find
\begin{align}\label{eq_near_sigma1}
\Delta_{\sigma_1+\psi}&(\sigma_1+\psi)-\varphi^*_{V(\sigma_1,\sigma_1+\psi)}\eta \notag
\\ &=\rho-\Delta_{\sigma_1}\psi+\mathcal L_{V_{\sigma_1,\sigma_1}'(\psi)}(\sign(\eta)\,\sigma_1-\eta) 
+\Psi_{\sigma_1,\sigma_1}(\psi)+J_{\sigma_1,\sigma_1}(\psi)
\end{align}
for $\psi\in B_q$. On the right-hand side,
\begin{align}\label{rho_def}
\rho=\Delta_{\sigma_1}\sigma_1-\eta,
\end{align}
and the map $J_{\sigma_1,\sigma_1}$ satisfies
\begin{align}\label{J_bound}
\|J_{\sigma_1,\sigma_1}(\psi)\|_{C^{1,a}\Omega^3(N_s)}&\le b\|\psi\|^2_{C^{3,a}\Omega^3(N_s)},\notag
\\
\|J_{\sigma_1,\sigma_1}(\psi_1)-J_{\sigma_1,\sigma_1}(\psi_2)\|_{C^{1,a}\Omega^3(N_s)}&\le bq\|\psi_1-\psi_2\|_{C^{3,a}\Omega^3(N_s)},\qquad \psi,\psi_1,\psi_2\in B_q,
\end{align}
with $b$ independent of $s$ and~$q$. We will now use~\eqref{eq_near_sigma1} to construct a 3-form~$\sigma_*$ satisfying~\eqref{eq_DeT=0}.

The equality $\sigma_1(p)=\sign(\eta)\,\eta(p)$ and~\eqref{coord_canon} imply that the highest-order coefficients of the linear differential operator
$$
\psi\mapsto\Delta_{\sigma_1}\psi-12^{-1}\Delta_g\psi-\mathcal L_{V_{\sigma_1,\sigma_1}'(\psi)}(\sign(\eta)\,\sigma_1-\eta)
$$
with respect to $(x_1,\ldots,x_7)$ vanish at~$p$. Let us apply Lemma~\ref{lem_lin_inv} with $\gamma=12^{-1}$ and $K$ given by
\begin{align*}
K\psi=
\Delta_{\sigma_1}\psi-12^{-1}\Delta_g\psi-\mathcal L_{V_{\sigma_1,\sigma_1}'(\psi)}(\sign(\eta)\,\sigma_1-\eta)&-\Psi_{\sigma_1,\sigma_1}(\psi).
\end{align*}
Shrinking $s$ if necessary, we obtain a bounded operator
$$P:C_0^{1,a}\mathcal X^3(N_s)\to C_0^{3,a}\mathcal X^3(N_s)$$
such that
\begin{align*}
\Delta_{\sigma_1}P\psi-\mathcal L_{V_{\sigma_1,\sigma_1}'(P\psi)}(\sign(\eta)\,\sigma_1-\eta)&-\Psi_{\sigma_1,\sigma_1}(P\psi)=\psi
\end{align*}
for all $\psi\in C_0^{1,a}\mathcal X^3(N_s)$.

Suppose that $q<\frac1{4b\|P\|}$.
By Lemma~\ref{lem_exten}, for the 3-form $\rho$ given by~\eqref{rho_def}, there exist $r\in(0,s)$ and $\hat\rho\in C_0^{1,a}\mathcal X^3(N_s)$ such that $\rho_{|N_r}=\hat\rho_{|N_r}$ and
$$
\|\hat\rho\|_{C^{1,a}\Omega^3(N_s)}<\frac q{2\|P\|}.
$$
Formulas~\eqref{J_bound} imply that the map 
$$
\mathfrak C:B_q\to B_q
$$
defined by
\begin{align*}
\mathfrak C(\psi)=P(\hat\rho+J_{\sigma_1,\sigma_1}(\psi)),\qquad\psi\in B_q,
\end{align*}
is a contraction on the closed ball~$B_q$. Denote the unique fixed point of this map by~$\psi_*$. Applying the operator $12^{-1}\Delta_g+K$ to both sides of the equality $\mathfrak C(\psi_*)=\psi_*$, we conclude that
\begin{align}\label{eq_psi_star_final}
\hat\rho-\Delta_{\sigma_1}\psi_*+\mathcal L_{V_{\sigma_1,\sigma_1}'(\psi_*)}(\sign(\eta)\,\sigma_1-\eta)
+\Psi_{\sigma_1,\sigma_1}(\psi_*)+J_{\sigma_1,\sigma_1}(\psi_*)=0.
\end{align}
In light of~\eqref{eq_near_sigma1} and the fact that $\rho_{|N_r}=\hat\rho_{|N_r}$, this yields formula~\eqref{eq_DeT=0} in $N_r$ with $\sigma_*=\sigma_1+\psi_*$.

Let us show that $\psi_*$ is smooth near~$p$. One may re-write~\eqref{eq_psi_star_final} as a system of second-order equations for the components $(\psi_{*ijk})_{i,j,k=1}^7$ of $\psi_*$ in the coordinates~$(x_1,\ldots,x_7)$. We differentiate each of these equations in $x_n$ for $n=1,\ldots,7$. As a result, noting that $\psi_*\in C_0^{3,a}\mathcal X^3(N_s)$, we find that $\big(\frac{\partial\psi_{*ijk}}{\partial x_n}\big)_{i,j,k=1}^3$ solve a quasilinear system that is elliptic near~$p$. Well-known regularity properties of such systems (see, e.g.,~\cite[Section~8.5]{LS68}) and a standard bootstrapping argument imply that $(\psi_{*ijk})_{i,j,k=1}^7$ are smooth.

Since $\psi_*\in C_0^{3,a}\mathcal X^3(N_s)$, the map $\varphi_{V(\sigma_1,\sigma_*)}$ does not move the point~$p$. In fact, this map is an orientation-preserving diffeomorphism between neighbourhoods of~$p$. Shrinking $r$ if necessary, one can easily conclude that the $G_2$-structure $\sigma$ given by~\eqref{sigma_gauged_final} is a smooth solution to~\eqref{eq_Poisson} in $U=\varphi_{V(\sigma_1,\sigma_*)}(N_r)$.

\end{document}